\documentclass[12pt]{article}
\usepackage{graphicx}
\usepackage{amsmath}
\usepackage{amsfonts}
\usepackage{amssymb}
\usepackage{url}
\newtheorem{theorem}{Theorem}[section]

\newtheorem{corollary}[theorem]{Corollary}

\newtheorem{defn}[theorem]{Definition}

\newtheorem{remark}[theorem]{Remark}

\numberwithin{equation}{section}

\newenvironment{proof}[1][Proof]{\textbf{#1.} }{\ \rule{0.5em}{0.5em}}

\usepackage{color}
\topmargin=-0.5in \everymath{\displaystyle}
\begin{document}
\baselineskip=18pt

\pagenumbering{arabic}

\begin{center}
{\Large {\bf Spatial Models of Vector-Host Epidemics with Directed Movement of Vectors Over Long Distances}}

\bigskip

W.E. Fitzgibbon and J.J. Morgan

Department of Mathematics

University of Houston\\
Houston, TX 77204, USA

\bigskip

Glenn F. Webb and Yixiang Wu

Department of Mathematics

Vanderbilt University\\
Nashville, TN 37212, USA

\bigskip

\vspace{0.2in}
\end{center}

\begin{abstract}
We investigate a time-dependent spatial vector-host epidemic model with non-coincident domains for the vector and host populations. The host population resides in small non-overlapping sub-regions, while the vector population resides throughout a much larger region. The dynamics of the populations are modeled by a reaction-diffusion-advection compartmental system of partial differential equations. The disease is transmitted through vector and host populations in criss-cross fashion. We establish global well-posedness and uniform a prior bounds as well as the long-term behavior. The model is applied to simulate the outbreak of bluetongue disease in sheep transmitted by midges infected with bluetongue virus. We show that the long-range directed movement of the midge population, due to wind-aided movement, enhances the transmission of the disease to sheep in distant sites.   
\end{abstract}

\medskip


\section{Introduction}
Many diseases, such as malaria, dengue fever, Zika, Chagas disease in humans, and bluetongue disease in sheep and other ruminants, are transmitted in criss-cross fashion between vectors and hosts. Vectors, such as mosquitoes, fleas, ticks, and midges, transmit microbial disease agents to animal and human host populations. Susceptible vectors  become infected by interaction with infected hosts, and infected vectors transmit the disease to susceptible hosts. Mathematical models for the spatial spread of such diseases have been developed by many authors, e.g. \cite{Allen2008, Anita2009, Capasso1978, Fitzgibbon1994, Fitzgibbon1995, Fitzgibbon1996, Fitzgibbon2004reaction, Fitzgibbon2005, Fitzgibbon19952, Fitzgibbon2008, fitzgibbon2017outbreak, Webb1981}. 


In this paper, we investigate a vector-host model describing the spatio-temporal spread of an epidemic disease, with 
host populations residing in non-overlapping small domains, and vectors residing in a much larger region. We assume that the host population has both a random spatial movement and a directed spatial movement.  We will show that the spread of a vector-host epidemic disease over a large geographic region can result from an outbreak in a relatively small host subregion, by directed long-range movement of vectors to distant host subregions.

The directed movement can result from external forces such as wind-aided movement. The role of wind-aided movement in the transmission of disease has been studied by many researchers, including \cite{Braverman, Burgin, Ducheyne, Murray1987A, Johansen, Sedda, Ssematimba}. In \cite{Ssematimba} the authors study the wind-borne transportation of Highly Pathogenic Avian influenza virus between farms, where the movement of pathogen particles is described by a Gaussian Plume Model, which is essentially an advection-diffusion model.   In \cite{Sedda}, the authors use stochastic simulations to study the impact of the movement of midges due to wind on the spread of bluetongue virus in Europe. In \cite{Burgin}, Burgin {\it et al.} the authors use an atmospheric dispersion model to study the impact of wind on the spread of bluetongue disease in sheep.

The organization of this paper is as follows: in Section 2 we analyze a general model of this class of epidemics, in which both vectors and hosts are diffusing in their domains;  in Section 3 we analyze a special case in which only the vectors are diffusing in their domain;  in Section 4 we apply the results in Section 3 to an outbreak of 
bluetongue disease in sheep. The results of Section 4 address the issue of the recent spread of bluetongue disease in Europe.

\section{The Model}

\subsection{The vector population model}
The vector habitat is considered to be a region sufficiently large to contain numerous sub regions which contain host populations.  We shall depart from the standard practice of using bounded regions to define species habitats and simply define the vector habitat as $\mathbb{R}^2$. We assume that the spatially distributed vector population $V(x, t)$  is subject to logistic demographics with spatially dependent linear birth term of the form $\beta(x)V(x,  t)$  and a quadratic self-limiting term of the form $-m(x)V^2(x, t)$.   Dispersion of the vectors is modeled by diffusion and advection with diffusion describing the natural movement of the vectors and advection accounting for the effective of the wind.  These considerations give rise to the following reaction-diffusion-advection type equation:
 \begin{equation}
\left\{
\begin{array} {lll}
   \frac{\partial V}{\partial t}-\triangledown\cdot D(x)\triangledown  V+\overrightarrow{C}(x, t) \cdot \triangledown V=\beta(x) V-m(x) V^2,   \label{vector}&\ \ \ x\in\mathbb{R}^2,\ t>0,\medskip\\
   V(x, 0)=V_0(x),  &\ \ \ x\in\mathbb{R}^2.
\end{array}
\right.
\end{equation}

Eq. \eqref{vector} is a classic semi-linear reaction-diffusion-advection equation commonly known as a convective Fisher-Kolmogorov equation. Fisher-Kolmogorov equations \cite{Fisher} first introduced in the 1930s remain of active interest and arise in a variety of applications. We assume that the diffusion coefficient $D(x)$ is up to order 2 uniformly bounded and continuous, strictly positive on $\mathbb{R}^2$, i.e.
\begin{itemize}
\item[A1.] $D\in C_b^2(\mathbb{R}^2)$;
\item[A2.] There exist $D_m, D_M>0$ such that $D_m\le D(x) \le D_M$ for all $x\in\mathbb{R}^2$.
\end{itemize}
 We make the following assumptions on the velocity field $\overrightarrow{C}(x,  t)$:
\begin{itemize}
\item[A3.]  $\overrightarrow{C}\in C(\mathbb{R}_+; C_b^1(\mathbb{R}^2))$;
\item[A4.] $\triangledown\cdot \overrightarrow{C}(\cdot, t)=0$ for all $t\ge 0$.
\end{itemize}
The terms $\beta(x)V$ and $m(x)V^2$ represent spatially dependent birth and logistic mortality rates of the vector population, respectively.  We require that
\begin{itemize}
\item[A5.] $\beta(x)\ge 0$ for all $x\in\mathbb{R}^2$, and $\beta\in C_b^1(\mathbb{R}^2)$ with $\|\beta\|_{\infty, \mathbb{R}^2}\le \beta^*<\infty$ for some positive constant $\beta^*$;
\item[A6.]  $m\in C_b^1(\mathbb{R}^2)$, and there exist positive constants $m_*, m^*$ such that $m_*\le m(x)\le m^*$ for all $x\in\mathbb{R}^2$.
\end{itemize}

We make the following assumption on the initial data:
\begin{itemize}
\item[A7.] $V_0(x)\ge 0$ for all $x\in \mathbb{R}^2$, and  $V_0\in C_b(\mathbb{R}^2)\cap L^1(\mathbb{R}^2)$.
\end{itemize}
Let $Q(s, t)$ be the space-time cylinder $\mathbb{R}^2\times (s, t)$.

The following theorem guarantees the global well-posedness of solutions to \eqref{vector}.
\begin{theorem}\label{theorem_vector}
Assume that (A1)-(A7) hold. Then there exists a unique classical global solution $V(x, t)$ of \eqref{vector} such that 
$$
0\le V(x, t)\le \max\left\{\frac{\beta^*}{m_*},  \|V_0\|_{\infty,   \mathbb{R}^2}\right\}, \ \ \ \text{for all } (x, t)\in Q(0, \infty).
$$
\end{theorem}
\begin{proof}
Local well-posedness can be obtained using a Green's function argument.  The uniform a priori bound and the non-negativity follow from the fact that  $(0, \beta^*/m_*)$ is an invariant rectangle \cite{Smoller1984shock}. The presence of the uniform a priori bound guarantees a global solution \cite{Amann1986quasi}.     
\end{proof}

 Throughout this paper, we suppose that $\Omega$ is a bounded domain in $\mathbb{R}^2$ with smooth boundary $\partial\Omega$ such that $\Omega$ lies locally on one side of $\partial\Omega$. We define an operator $T: H^1(\Omega)\rightarrow H^1(\Omega)$ by 
$$
T(u):=\triangledown\cdot D\triangledown u+\beta u, \ u\in H^2(\Omega)\cap H_0^1(\Omega),
$$
It is well-known that the principal eigenvalue of $T$ has the following variational characterization:
$$
\max_{u\in H_0^1(\Omega), \|u\|_{2, \Omega}=1} \int_{\Omega} (-D(x)|\triangledown u(x)|^2 +\beta(x)u^2(x))dx.
$$

We can prove the following:
\begin{theorem}\label{theorem_uniform}
Suppose that (A1)-(A2) and (A5)-(A6) hold, $\overrightarrow{C}=0$, and $V_0$ is nontrivial. If the principal eigenvalue of $T$ is positive, then for any $\Omega'\subset\subset\Omega$, there exist $\epsilon_0>0$ independent on $V_0$ such that the solution $V(x, t)$ of  \eqref{vector}  satisfies $V(x,  t)>\epsilon_0$ for all $x\in \bar\Omega'$ and $t>t_0$ for some $t_0>0$.
\end{theorem}
\begin{proof}
Since the principal eigenvalue of $T$ is positive, there is a unique positive solution $u$ of 
\begin{equation}
\left\{
\begin{array}{lll}
\triangledown\cdot D\triangledown u+\beta u-mu^2=0, \ \ &x\in \Omega,\\
u=0, \ \  &x\in\partial \Omega. 
\end{array}
\right.
\end{equation}
For any $v_0\in C_0(\bar\Omega)$ that is  nonnegative and nontrivial, let $v(x, t)$ be the solution of the problem 
\begin{equation}\label{comp1}
\left\{
\begin{array}{lll}
v_t=\triangledown\cdot D\triangledown v+\beta v-mv^2, \ \ &x\in \Omega, t>0,\\
v(x, t)=0, \ \  &x\in\partial \Omega, t>0,\\
v(x, 0)=v_0, \ \ &x\in \Omega.
\end{array}
\right.
\end{equation}
Then, we have (see, e.g. \cite{Cantrell2004})
\begin{equation}\label{lim1}
\lim_{t\rightarrow\infty} \|v(\cdot, t)- u\|_{\infty, \Omega}=0.
\end{equation}

Suppose that $V_0$ is nontrivial. Then by the maximum principle, $V(x, t)>0$ for all $(x,  t)\in Q(0, \infty)$. So without loss of generality, we may assume that $V_0(x)>0$ for all $x\in \Omega$. Let $v_0\in C_0(\bar\Omega)$ such that $0<v_0(x)<V_0(x)$ for all $x\in\Omega$. Then $V(x, t)$ is a super solution of \eqref{comp1}. By the comparison principle, we have $V(x, t)>v(x, t)$ for all $x\in\Omega$ and $t>0$. Let $a=\min_{x\in \overline{\Omega'}}  u(x)$. By \eqref{lim1}, there exists $t_0>0$ such that 
$$
V(x, t)>v(x,  t)>\frac{a}{2}:=\epsilon_0, \ \text{ for all } \ x\in \overline{\Omega'} \ \text{ and } \ t>t_0.
$$
\end{proof}

For $r>0$, let $B_r$ be the closed disk in $\mathbb{R}^2$ of radius $r$ centered at the origin.
\begin{corollary}\label{corollary_uniform}
Suppose that (A1)-(A2) and (A5)-(A6) hold and $\overrightarrow{C}=0$. If
\begin{equation}\label{beta1}
\int_{\mathbb{R}^2} \beta(x)dx>\frac{\pi^2 D_{M}}{2},
\end{equation}
then for any $r>0$, there exist $\epsilon_0>0$ independent on $V_0$ and $t_0>0$ such that the solution $V(x, t)$ of  \eqref{vector} satisfies $V(x, t)>\epsilon_0$ for all $x\in B_r$ and $t>t_0$.
\end{corollary}
\begin{proof}
For each $a>0$, let $X_a=(-a, a)\times (-a, a)$. Let $\epsilon\in (0, 1)$ be given. By the assumption \eqref{beta1}, there exits $K>0$ such that 
$$
(1-\epsilon)\int_{X_K} \beta(x)dx>\frac{\pi^2 D_{M}}{2},
$$
and $B_r$ is a  subset of $X_K$. In addition, there exists $\delta=\delta(\epsilon)\in (0, 1)$ such that $1-\epsilon\le \sin^2(w)\le 1$ whenever $|w-\pi/2|\le \delta$. Let $L=\frac{\pi}{2\delta}K$. Then $L>K$ and $B_r$ is a subset of $X_L$. Moreover, one can check that $1-\epsilon \le \sin^2(\frac{\pi (x+L)}{2L})\le 1$ whenever $|x|\le K$.

Define 
$$
\phi(x)=\frac{1}{L} \sin(\frac{\pi(x_1+L)}{2L}) \sin(\frac{\pi(x_2+L)}{2L}) \ \text{ for } x=(x_1, x_2)\in X_L. 
$$
Then $\phi\in H^1_0(X_L)$, $\|\phi\|_{2, X_L}=1$ and $\|\triangledown \phi\|_{2, X_L} =\frac{\pi}{\sqrt{2}L}$. It then follows that 
\begin{eqnarray*}
&&\max_{u\in H_0^1(X_L), \|u\|_{2, X_L}=1} \int_{X_L} (-D(x)|\triangledown u(x)|^2 +\beta(x)u^2(x))dx\\
&&\hspace{3cm} \ge \int_{X_L} (-d_{max}|\triangledown \phi(x)|^2 + \beta(x)\phi^2(x))dx\\
&&\hspace{3cm} \ge -D_M\frac{\pi^2}{2L^2}+\int_{X_K} \beta(x)\phi^2(x)dx  \\
&&\hspace{3cm} \ge -D_M\frac{\pi^2}{2L^2}+\frac{1-\epsilon}{L^2} \int_{X_K} \beta(x)dx>0.
\end{eqnarray*}
Consequently, the principal eigenvalue of $T$ is positive, and the claim follows from Theorem \ref{theorem_uniform}.
\end{proof}

%
%

\subsection{Vector host transmission}
We assume that the host population is distributed between two distinct bounded sub-regions   $\Omega_1$ and $\Omega_2$ of $\mathbb{R}^2$ that are in sufficiently close proximity to allow natural vector diffusion without the presence of the wind to drive the vector borne transmission of the pathogen from one field to another.  Both  $\Omega_1$ and $\Omega_2$ are assumed to have smooth boundaries $\partial\Omega_1$ and $\partial\Omega_2$ and to lie locally on one side of their boundaries.  The sub-regions are non-overlapping and separated:
$$
\Omega_1\cap\Omega_2 =\emptyset \ \ \text{ and } \ \ dist(\Omega_1, \Omega_2)>0.
$$

We let $H_1$ and $H_2$ denote the host populations which occupy $\Omega_1$  and $\Omega_2$ respectively, and assume that  $H_1$ remains confined to $\Omega_1$ and $H_2$ remains confined to $\Omega_2$.  We model the circulation in each of the subregions by an SEIR model. The susceptible class, $S_j$ for $j=1, 2$, consists of individuals who are free of the pathogen. The exposed class, $E_j$ for $j=1, 2$, consists of individuals who have been infected with the pathogen. However at this stage the disease is incubating and these individuals are not capable of transmitting the pathogen. The infected/infectious class, $I_j$ for $j=1, 2$, consists of individuals who are capable of transmitting the disease. The removed class, $R_j$, for $j=1, 2$, consists of individuals who have either perished from the disease or have recovered, and  thereby gained immunity. The variables $S_j(x, t)$, $E_j(x, t)$, $I_j(x, t)$, $R_j(x, t)$ for $j=1, 2$ represent the time dependent spatial densities in each of the subregions $\Omega_j$. The total population of each class in each subregion is given by integration over the subregion. 


Susceptible hosts in each subregion are infected via contact with infected vectors. We model this by mass action force infection terms: $\sigma_1(x)S_1(x, t)V_i(x, t)$, $x\in\Omega_1, t>0$,  and $\sigma_2(x)S_2(x, t)V_i(x, t)$, $x\in\Omega_2, t>0$. We assume that exposed hosts in either subregion become fully infected at constant rate  $\lambda>0$, and that removal by death or recovery in either subregion occurs at a rate $\delta>0$.

The host population of each subregion remains confined to that subregion. The dispersion through each subregion is modeled by diffusion with the diffusivities of the susceptible and exposed hosts in subregions  $\Omega_1$ and $\Omega_2$ given by $D_{11}(x)$ and $D_{12}(x)$.  The dispersion of infected/infective hosts in $\Omega_1$  and $\Omega_2$ is modeled by $D_{21}(x)$ and $D_{22}(x)$.

Infected vectors can be recruited by means of contact with infected hosts in either of the two sub-regions.  This process is modeled by the incidence term:
\begin{equation*}
f(x, t, I_{1}, I_{2}, V_s)=
\left\{
\begin{array}{lrl}
\alpha_{1}(x) I_{1}V_s, \ \ &x\in\bar\Omega_1, t>0  \\
\alpha_{2}(x) I_{2}V_s, \ \ &x\in\bar\Omega_2, t>0 \\
0, \ \ \ &x\in\mathbb{R}^2/\bar\Omega_1\cup \bar\Omega_2, t>0. 
\end{array}
\right.
\end{equation*}
We assume that the presence of the pathogen has no deleterious effect on the vectors.

The following equations model the vector-host populations:
\begin{itemize}
\item \textbf{Vector Populations}
\begin{equation}\label{main1}
\left\{
\begin{array}{lll}
\frac{\partial}{\partial t} V_s =  \triangledown\cdot D(x)\triangledown V_s-\overrightarrow{C}(x, t)\cdot \triangledown  V_s+\beta(x) 
V-m(x) V_s V- f(x, t, I_1, I_2, V_s), \\
\hspace{4.5in} (x, t)\in Q(0, \infty), \\ 
\frac{\partial}{\partial t} V_i =  \triangledown\cdot D(x)\triangledown V_i-\overrightarrow{C}(x, t)\cdot \triangledown  V_i-m(x) V_i V
+ f(x, t, I_1, I_2, V_s),  \\
\hspace{4.5in} (x, t)\in Q(0, \infty),  \\
V_s(x, 0)=V_{s0}(x), \ \ \ V_i(x, 0)=V_{i0}(x), \ \ \ x\in\mathbb{R}^2.
\end{array}
\right.
\end{equation}

\item \textbf{Host Populations}

\begin{equation}\label{main2}
\left\{
\begin{array}{lll}
\frac{\partial}{\partial t} S_{j} =  \triangledown\cdot D_{1j}(x)\triangledown S_{j}-\sigma_j(x) S_jV_i, \ \ \ (x, t)\in \Omega_j\times (0, \infty), \, j=1,2, \medskip \\ 
\frac{\partial}{\partial t} E_{j} = \triangledown\cdot D_{1j}(x)\triangledown E_{j}+\sigma_j(x) S_jV_i-\lambda E_j , \ \ \ (x, t)\in \Omega_j\times (0, \infty),  \, j=1,2, \medskip \\ 
\frac{\partial}{\partial t} I_{j} = \triangledown\cdot D_{2j}(x)\triangledown I_{j}+\lambda  E_j-\delta I_j , \ \ \ (x, t)\in \Omega_j\times (0, \infty), \, j=1,2, \medskip \\ 
\frac{\partial }{\partial n} S_{j}=\frac{\partial }{\partial n} E_{j}=\frac{\partial }{\partial n} I_{j}=0, \ \ \ (x, t)\in \partial\Omega_j\times (0, \infty), \, j=1,2, \medskip \\ 
S_{j}(x, 0)=S_{j0}(x), \ \ \ E_{j}(x, 0)=E_{j0}(x),\ \ \ I_{j}(x, 0)=I_{j0}(x), \ \ \ \ x\in\Omega_j, \, j=1,2.
\end{array}
\right.
\end{equation}
\end{itemize}
The removed classes have no effect on the progress of the disease and do not appear in our system of equations. The homogeneous Neumann boundary conditions in \eqref{main2} guarantee that the hosts remain confined to their habitats. 

We further impose the following assumptions:
\begin{itemize}
\item[A8.]   $D_{11}, D_{21}\in C_b^2(\bar\Omega_1)$, $D_{12}, D_{22}\in C^2_b(\bar\Omega_2)$, and there exists positive constant $D_*$ such that $D_{11}, D_{21}, D_{12}, D_{22}\ge D_*$;
\item[A9.] $\lambda, \delta>0$;
\item[A10.] $\sigma_j, \alpha_j\in C^1_b(\bar\Omega_j)$ and there exist $\sigma_*, \alpha_*$ such that $\sigma_j(x)\ge \sigma_*$ and  $\alpha_j(x)\ge \alpha_*$ for all $x\in \bar\Omega_j$, $j=1, 2$; 
\item[A11.] $S_{j0}, E_{j0}, I_{j0}\in C(\bar\Omega_j)$, and $S_{j0}(x), E_{j0}(x), I_{j0}(x)\ge 0$ for all $x\in\bar\Omega_j$, $j=1, 2$;
\item[A12.]   $V_{s0}, V_{i0}\in  C_b(\mathbb{R}^2)\cap L^1(\mathbb{R}^2)$, and $V_{s0}(x), V_{i0}(x)\ge 0$ for all $x\in \mathbb{R}^2$.
\end{itemize}

We remark that the discontinuity produced by the left hand side of the equations for $V_s$  and $V_i$  precludes a classical  global existence theorem. 
\begin{defn}
We say $(V_s(x, t), V_i(x, t))$, $(x, t)\in Q(0, \infty)$, and $(S_{j}(x, t), E_{j}(x, t), I_{j}(x, t))$, $(x, t)\in\Omega_j\times (0, \infty)$, $j=1, 2$, are strong global solutions of \eqref{main1}-\eqref{main2}, if
\begin{itemize}
\item $S_{j}, E_{j}, I_{j}\in C^{2, 1}(\bar\Omega_j\times (0, \infty))$, $j=1, 2$;
\item $V_s, V_i\in C((0, \infty); C_b(\mathbb{R}^2 )\cap L^1(\mathbb{R}^2))$;
\item For each $p>1$, $V_s, V_i\in C((0, \infty); W^{2, p}(\mathbb{R}^2))$; 
\item The partial differential equations and initial conditions are a.e. satisfied.
\end{itemize}
\end{defn}

Our well-posedness result for the vector-host model is as follows:
\begin{theorem}\label{existence}
Assume that (A1)-(A6) and (A8)-(A12) hold. Then there is a unique nonnegative global bounded strong solution of \eqref{main1}-\eqref{main2}.
\end{theorem}

\begin{proof}
We can adapt the Green's function/variation of parameters method to establish the local well-posedness on a maximal time interval $[0, T_{max})$ with $T_{max}=\infty$ provided the supreme norm does not blow up in finite time. 

Since we have assumed that all initial conditions are non-negative, we can adapt standard invariant rectangle arguments \cite{Smoller1984shock} to observe that all solution components remain non-negative.  By Theorem \ref{theorem_vector} and $V(x, t)=V_s(x, t)+V_i(x, t)$, we have 
$$
0\le V_s(x, t), V_i(x, t)\le M,
$$
where $M$ is some positive constant depending on the initial data $V_{s0}, V_{i0}$.  By the equations of \eqref{main2} and the comparison principle, we have 
$$
0\le S_j(x, t) \le \|S_{j0}\|_{\infty, \Omega_j} \text{ for all }  (x, t)\in \bar\Omega_j\times [0, T_{max}), \  j=1, 2.
$$
Then by the equations of \eqref{main2}, we have 
$$
\frac{\partial}{\partial t} E_{j} \le \triangledown\cdot D_{1j}(x)\triangledown E_{j}+M_1-\lambda E_j,
$$
where $M_1$ is some positive constant depending on $V_{s0}, V_{i0}$, $S_{10}$ and $S_{20}$. Hence $E_j$ is a lower solution of the problem 
\begin{equation*}
\left\{
\begin{array}{lll}
\frac{\partial}{\partial t} \bar E_{j} = \triangledown\cdot D_{12}(x)\triangledown \bar E_{j}+M_1-\lambda \bar E_j , \ \ \ &(x, t)\in \Omega_j\times (0, T_{max}), \medskip \\ 
\frac{\partial }{\partial n} \bar E_{j}=0, \ \ \ &(x, t)\in \partial\Omega_j\times (0, T_{max}), \medskip \\ 
\bar E_{j}(x, 0)=E_{j0}, \ \ \ &x\in\Omega_j.
\end{array}
\right.
\end{equation*}
By the comparison principle, then we have 
$$
0\le E_j(x, t)\le \bar E_j(x, t) \le \max\{ M_1/\lambda,  \|E_{j0}\|_{\infty, \Omega_j}\},  \ \ (x, t)\in \Omega_j\times (0, T_{max}), j=1, 2.
$$
Similarly, by the equations of \eqref{main2}, we can obtain similar bounds for $I_j$, $j=1, 2$. Therefore, $T_{max}=\infty$, and we have global boundedness of the solution. 
\end{proof}

\begin{theorem}\label{theorem_asymptotic}
Assume that (A1)-(A6) and (A8)-(A12) hold.  Then there exist nonnegative constants $S^*_{1}$ and $S^*_{2}$ such that 
\begin{eqnarray}
&&\lim_{t\rightarrow\infty} \|S_{j}(\cdot, t)- S^*_{j}\|_{\infty, \Omega_j}=0, \ j=1, 2 \\
&&\lim_{t\rightarrow\infty} \|E_{j}(\cdot, t)\|_{\infty, \Omega_j}=0, \ j=1, 2 \\
&&\lim_{t\rightarrow\infty} \|I_{j}(\cdot, t)\|_{\infty, \Omega_j}=0, \ j=1, 2\\
&&\lim_{t\rightarrow\infty} \|V_{i}(\cdot, t)\|_{\infty, \mathbb{R}^2}=0.
\end{eqnarray} 
Furthermore, if $\overrightarrow{C}=0$, $V_{s0}+V_{i0}$ and $S_{j0}$ are nontrivial, and  
\begin{equation}\label{beta111}
\int_{\mathbb{R}^2} \beta(x)dx>\frac{\pi^2 D_{M}}{2},
\end{equation}
then $S^*_j>0$, $j=1, 2$.
\end{theorem}
\begin{proof}
Adding up the equations in \eqref{main2}  (add up the first two equations twice) and integrating them over $\Omega_j$, we obtain
$$
\frac{d}{dt} \int_{\Omega_j} (2S_{j}(x, t)+2E_{j}(x, t)+ I_{j}(x, t)) dx+\lambda \int_{\Omega_j} E_{j}(x, t)dx+\delta \int_{\Omega_j} I_{j}(x, t)dx\le 0, \ \ j=1, 2.
$$
Integrating the equation with respect to time on $[0, t]$, we have
\begin{equation}\label{sib0}
\lambda \int_0^t\int_{\Omega_j} E_{j}(x, s)dxds+\delta\int_0^t\int_{\Omega_j} I_{j}(x, s)dxds\le \int_{\Omega_j} (S_{j0}(x)+E_{j0}(x)+I_{j0}(x)) dx,  \ \ j=1, 2.
\end{equation}
We thereby may conclude that for $\tau>0$,
$$
\lim_{\tau\rightarrow\infty}\int_\tau^{\tau+1}\int_{\Omega_j}  E_{j}(x, s)dxds=0 \ \ \ \text{ and } \ \ \ \lim_{\tau\rightarrow\infty}\int_\tau^{\tau+1}\int_{\Omega_j}  I_{j}(x, s)dxds=0, \ \ j=1, 2.
$$
This together with the uniform a priori bounds on $E_j$ and $I_j$ implies that for $p\ge 1$
\begin{equation}\label{l1}
\lim_{\tau\rightarrow\infty} \|E_j\|_{p, \Omega_j\times (\tau, \tau+1)}= \lim_{\tau\rightarrow\infty} \|I_j\|_{p, \Omega_j\times (\tau, \tau+1)}=0, \ \ j=1, 2. 
\end{equation}
By the equations of \eqref{main2} and the parabolic $L^p$ estimate, there exists $C_p>0$ such that for $\tau>0$
$$
 \|I_j\|^{(2, 1)}_{p, \Omega_j\times (\tau+1, \tau+2)}\le C_p( \|E_j\|_{p, \Omega_j\times (\tau, \tau+2)}+  \|I_j\|_{p, \Omega_j\times (\tau, \tau+2)} ), \ \ j=1, 2.
$$
Consequently, by \eqref{l1} and the Sobolev embedding theorem, we conclude that 
$$
\lim_{t\rightarrow\infty} \|I_j(\cdot, t)\|_{\infty, \Omega_j} =0, \ \ j=1, 2. 
$$

We now turn our attention to $E_j$. Integrating the equation for $E_j$ results in for every $t>0$
$$
\int_0^t \int_{\Omega_j} \sigma_j(x) S_j(x, \tau)V_i(x, \tau) dxd\tau \le \int_{\Omega_j} S_{j0}(x)dx, \ \ j=1, 2.
$$
Therefore the boundedness of $\sigma_j, S_j, V_i$ implies that for every $p>1$,
$$
\int_0^\infty\int_{\Omega_j} (\sigma_j(x) S_j(x, \tau)V_i(x, \tau))^p dxd\tau <\infty \ \ j=1, 2.
$$
As a result, 
\begin{equation}\label{sss}
\lim_{\tau\rightarrow\infty} \|\sigma_jS_jV_i\|_{p, \Omega_j\times (\tau, \tau+1)}=0.
\end{equation}
By the equations of \eqref{main2} and the parabolic $L^p$ estimate, there exist $C_p>0$ such that for $\tau>0$
$$
 \|E_j\|^{(2, 1)}_{p, \Omega_j\times (\tau+1, \tau+2)}\le C_p( \|E_j\|_{p, \Omega_j\times (\tau, \tau+2)}+  \|\sigma_jS_jV_i\|_{p, \Omega_j\times (\tau, \tau+2)} ), \ \ j=1, 2.
$$
Again by \eqref{l1}-\eqref{sss} and the Sobolev embedding theorem, we conclude that 
$$
\lim_{t\rightarrow\infty} \|E_j(\cdot, t)\|_{\infty, \Omega_j} =0, \ \ j=1, 2. 
$$

We now turn our attention to $V_i$. Recall the definition of the incidence function
$$
F(x, t)=: f(x, t, I_1(x, t), I_2(x, t), V_s(x, t)).
$$
By the boundedness of the solution and \eqref{sib0}, we observe that 
$$
\int_0^\infty \int_{\mathbb{R}^2} F(x, t) dx dt<\infty. 
$$
Using (A4), we integrate the equation for $V_i$ and observe that 
\begin{equation*}
\int_{\mathbb{R}^2} V_i(x, t)dx+ \int_0^t \int_{\mathbb{R}^2} m(x)V(x, \tau)V_i(x, \tau)dx dt=\int_{\mathbb{R}^2} V_{i0}(x)dx+\int_0^t\int_{\mathbb{R}^2} F(x, t) dx dt.
\end{equation*}
Therefore,
\begin{equation}\label{vi}
\int_0^\infty \int_{\mathbb{R}^2} m(x)V(x, \tau)V_i(x, \tau)dx dt <\infty.
\end{equation}
Since $0\le V_i(x, t)\le V(x, t)$ and $m(x)\ge m_*>0$, we have
$$
\int_0^\infty \int_{\mathbb{R}^2} V_i^2(x, \tau)dx dt <\infty.
$$
Consequently, 
\begin{equation*}
\lim_{\tau\rightarrow\infty} \|V_i\|_{2, \mathbb{R}^2\times (\tau, \tau+1)}=0.
\end{equation*}
Let $g(x, t)=F(x, t)-m(x)V(x, t)V_i(x, t)$. We rewrite the equation for $V_i$ as 
$$
\frac{\partial V_i}{\partial t}=\triangledown\cdot D(x)\triangledown V_i - \overrightarrow{C}\cdot \triangledown V_i + g(x, t).
$$
We then can adapt the argument in \cite[Theorem 4.1]{Fitzgibbon2004reaction} to insure that 
$$
\lim_{t\rightarrow\infty} \|V_i(\cdot, t)\|_{\infty, \mathbb{R}^2}=0.
$$

We now examine the convergence of $S_j$. Multiplying both sides of the equation for $S_j$ by $S_j(x, t)$ and integrating over $(0, t)\times \Omega_j$, we obtain
$$
\int_{\Omega_j} S_j^2(x, t) dx + 2\int_0^t\int_{\Omega_j} D_{1j}(x)|\triangledown S_j(x, \tau)|^2dxd\tau \le \int_{\Omega_j} S_{j0}^2(x) dx, \ \ j=1, 2.
$$
This implies 
\begin{equation}\label{l2}
\int_0^\infty\int_{\Omega_j} |\triangledown S_j(x, t)|^2dxdt<\infty. 
\end{equation}
Multiplying both sides of the equation for $S_j$ by $\partial S_j/\partial t$ and integrating over $\Omega_j$, we obtain
$$
\int_{\Omega_j} ( \frac{\partial S_j}{\partial t} )^2 dx + \frac{\partial}{\partial t} \int_{\Omega_j} D_{1j} |\triangledown S_j|^2 dx = - \int_{\Omega_j} \frac{\partial S_j}{\partial t} S_j V_idx, \ \ j=1, 2. 
$$
By Young's inequality, we have 
$$
\frac{3}{4}\int_{\Omega_j} ( \frac{\partial S_j}{\partial t} )^2 dx + \frac{\partial}{\partial t} \int_{\Omega_j} D_{1j} |\triangledown S_j|^2 dx \le \int_{\Omega_j}  S^2_j V^2_idx, \ \ j=1, 2. 
$$
Hence, there exists $M>0$ such that 
$$
\frac{\partial}{\partial t} \int_{\Omega_j} |\triangledown S_j|^2 dx  < M \text{ for all } \ t>0, \ \ j=1, 2. 
$$
Using \eqref{l2}, we have 
\begin{equation}\label{ll2}
\lim_{t\rightarrow\infty} \|\triangledown S_j(\cdot, t)\|_{2, \Omega_j}=0, \ \ j=1, 2.
\end{equation}
Integrating both sides of the equation for $S_j$ over $\Omega_j$, we can see that 
$$
\frac{d}{dt} \int_{\Omega_j} S_j(x, t)dx \le 0, \ \ j=1, 2. 
$$
Hence, there exists nonnegative constant $S^*_j$ such that 
$$
\frac{1}{|\Omega_j|} \int_{\Omega_j} S_j(x, t)dx\rightarrow S^*_j \ \text{as } t\rightarrow\infty, \ \ j=1, 2. 
$$
It then follows from the Poincare's inequality and \eqref{ll2} that
$$
\lim_{t\rightarrow \infty} \|S_j(\cdot, t)-S_j^*\|_{2, \Omega_j}=0, \ \ j=1, 2.
$$
Then by a standard bootstrapping argument, we have 
\begin{equation}\label{ss}
\lim_{t\rightarrow \infty} \|S_j(\cdot, t)-S_j^*\|_{\infty, \Omega_j}=0, \ \ j=1, 2.
\end{equation}

Now suppose that \eqref{beta111} holds.
By Corollary \ref{corollary_uniform}, there exist $\epsilon_0>0$ and $t_0>0$ such that 
$$
V(x, t)>v(x, t)>\epsilon_0, \ \text{ for all } \ x\in \overline{\Omega_1\cup\Omega_2} \ \text{ and } \ t>t_0.
$$
It then follows from \eqref{vi} that 
\begin{equation}\label{vii}
\int_0^\infty \int_{\Omega_j} V_i(x, t)dxdt<\infty, \ \ j=1, 2.
\end{equation}

Finally, we show $S_j^*>0$. Since $S_{10}$ and $S_{20}$ are nontrivial, $S_j(x, t)>0$ for all $x\in\bar\Omega_j$ and $t>0$ by the comparison principle. Without loss of generality, we may assume $S_{j0}(x)>0$ for  all $x\in\bar\Omega_j$. Then we can choose $\epsilon$ small such that $S_{j0}(x)>\epsilon$ for  all $x\in\bar\Omega_j$, $j=1, 2$. Define $U_j(x, t)=S_j(x, t)-\ln (S_j(x, t))$ for $x\in \Omega_j$ and $t\ge 0$.
By 
\begin{eqnarray*}
 -\frac{\partial}{\partial t} \ln (S_j)&=&-\frac{1}{S_j}\frac{\partial S_j}{\partial t}=-\frac{\triangledown\cdot D_{1j}\triangledown S_j-\sigma_j S_j V_i}{S_j}\\
&=&-\triangledown\cdot D_{1j}\triangledown \ln(S_j)-\frac{D_{1j}|\triangledown S_j|^2}{S_j^2}+\sigma_jV_i\\
&\le&-\triangledown\cdot D_{1j}\triangledown \ln(S_j)+\sigma_jV_i,
\end{eqnarray*}
 we have
 \begin{equation*}
 \left\{
\begin{array}{lll}
\frac{\partial U_j}{\partial t}\le \triangledown\cdot D_{1j}\triangledown U_j + h_j(x, t), \ \ \ &x\in\Omega_j, t>0, \medskip \\
\frac{\partial U_j}{\partial n}=0, \ \ \ &x\in\partial\Omega_j, t>0, \medskip\\
U_j(x, 0)\le \|S_{j0}\|_{\infty, \Omega_j}-\ln(\epsilon), \ \ \ &x\in\Omega_j,
\end{array} 
 \right.
 \end{equation*}
 with $h_j=-\sigma_jS_j V_i+\sigma_jV_i$, $j=1, 2$.
Using \eqref{vii}, we know 
 $$
 \int_0^\infty \int_{\Omega_j} h_j(x, t)dxdt<\infty.
 $$
By the comparison principle, we have
$$
\int_{\Omega_j} U_j(x, t) dx\le  |\Omega_j| (\|S_{j0}\|_{\infty, \Omega_j}-\ln(\epsilon)) +  \int_0^\infty \int_{\Omega_j} h_j(x, t)dxdt<\infty , \ \ t>0, j=1, 2.
$$ 
By virtue of \eqref{ss}, we observe that 
$$
S_j^*-\ln(S_j^*)<\infty, 
$$
which implies $S_j^*>0$, $j=1, 2$.
\end{proof}

We remark that if there exists constant $\beta_*>0$ such that $\beta(x)\ge \beta_*$ for all $x\in\mathbb{R}^2$ then \eqref{beta111} holds and the conclusion of Theorem \ref{theorem_asymptotic} is true.

\begin{remark}
We point out that the
analytical arguments of this section are readily extendable to handle the case of a diffusing host in each of the subregions. However, the point of this section is demonstrate the spread of the disease over much larger region which contains numerous relatively small subregions not to analyze the local dynamics among subregions in close proximity to one another.
\end{remark}

\section{A Special Case With Hosts Not Diffusing}

In this section, we will focus upon the advective diffusive spread of vector borne disease over a large region, where the host species is confined to multiple isolated subregions. Since our interest is region wide, we shall not be concerned with the spatial dynamics of the hosts within the subregions.  We consider distinct sub-regions $\Omega_j$ of $\mathbb{R}^2$ for $j=1, 2, ..., N$ with smooth boundaries $\partial \Omega_j$, such that $\Omega_j$ lies locally on one side of $\partial\Omega_j$ for each $j=1, 2, ..., N$. The sub-regions are non-overlapping and separated:
$$
\Omega_j\cap\Omega_k\neq\emptyset \ \ \ \text{ and } \ \ \ \text{dist}(\bar\Omega_j, \bar\Omega_k)>0, \ \ \text{ for }\  j, k =1, 2, ..., N\ \text{ with } \ j\neq k. 
$$
The circulation of the pathogen in each of the subregions is described by a spatially distributed non-diffusive SEIR model with compartments $S_j(x, t), E_j(x, t)$, and $I_j(x, t)$ for $j=1, 2, ..., N$. Again we need not consider the removed classes $R_j(x, t)$. Susceptible hosts in each subregion are infected via contact with infected vectors, which is modeled by $\sigma_jS_j V_i$. Infected vectors can recruited by means of contact with infected hosts in any of the sub-regions, and this process is modeled by the incidence term:
\begin{equation*}
f(x, t, I_{1},..., I_{N}, V_s)=
\left\{
\begin{array}{lrl}
\alpha_{j}(x) I_{j}V_s, \ \ &x\in\bar\Omega_j, t>0  \\
0, \ \ \ &x\in\mathbb{R}^2/\cup_{j=1}^N\bar\Omega_j, t>0. 
\end{array}
\right.
\end{equation*}

The N-subregions model is as follows:
\begin{itemize}
\item \textbf{Vector Populations}
\begin{equation}\label{main11}
\left\{
\begin{array}{lll}
\frac{\partial}{\partial t} V_s =  \triangledown\cdot D(x)\triangledown V_s-\overrightarrow{C}(x, t)\cdot \triangledown  V_s+\beta(x) V-m(x) V_s V- f(x, t, I_1, ... I_N, V_s),  \\
\hspace{4.5in} (x, t)\in Q(0, \infty), \\  
\frac{\partial}{\partial t} V_i =  \triangledown\cdot D(x)\triangledown V_i-\overrightarrow{C}(x, t)\cdot \triangledown  V_i-m(x) V_i V+ f(x, t, I_1, ..., I_N, V_s),\\
\hspace{4.5in} (x, t)\in Q(0, \infty),  \medskip \\
V_s(x, 0)=V_{s0}(x), \ \ \ V_i(x, 0)=V_{i0}(x), \ \ \ x\in\mathbb{R}^2.
\end{array}
\right.
\end{equation}

\item \textbf{Host Populations}
\begin{equation}\label{main22}
\left\{
\begin{array}{lll}
\frac{\partial}{\partial t} S_{j} =  -\sigma_j(x) S_jV_i, \ \ \ (x, t)\in \Omega_j\times (0, \infty), j=1, 2, ..., N, \medskip \\   
\frac{\partial}{\partial t} E_{j} = \sigma_j(x) S_jV_i-\lambda E_j , \ \ \ (x, t)\in \Omega_j\times (0, \infty), j=1, 2, ..., N, \medskip \\ 
\frac{\partial}{\partial t} I_{j} = \lambda  E_j-\delta I_j , \ \ \ (x, t)\in \Omega_j\times (0, \infty), j=1, 2, ..., N, \medskip \\ 
S_{j}(x, 0)=S_{j0}(x), \ \ \ E_{j}(x, 0)=E_{j0}(x),\ \ \ I_{j}(x, 0)=I_{j0}(x), \ \ \ x\in\Omega_j, , j=1, 2, ..., N.
\end{array}
\right.
\end{equation}
\end{itemize}

The hypotheses (A1)-(A7) and (A9)-(A12) are the same except that $j=1, 2$ is replaced by $j=1, 2, ..., N$. We modify our notion of a classical strong solution:
\begin{defn}
We say $(V_s(x, t), V_i(x, t))$, $(x, t)\in Q(0, \infty)$, and $(S_{j}(x, t), E_{j}(x, t), I_{j}(x, t))$, $(x, t)\in\Omega_j\times (0, \infty)$, $j=1, N$, are strong global solution of \eqref{main11}-\eqref{main22}, if
\begin{itemize}
\item $S_{j}, E_{j}, I_{j}\in C^{0, 1}((0, \infty)\times\bar\Omega_j)$, $j=1,2,...,N$;
\item $V_s, V_i\in C((0, \infty); C_b(\mathbb{R}^2 )\cap L^1(\mathbb{R}^2))$;
\item For each $p>1$, $V_s, V_i\in C((0, \infty); W^{2, p}(\mathbb{R}^2))$; 
\item The partial differential equations and initial conditions are a.e. satisfied.
\end{itemize}
\end{defn}

We have the following well-posedness result:
\begin{theorem}\label{existence1}
Assume that (A1)-(A7) and (A9)-(A12) hold. Then there is a unique nonnegative global bounded strong solution of \eqref{main11}-\eqref{main22}.
\end{theorem}

We then establish the following result about the global asymptotic behavior of the solutions of \eqref{main11}-\eqref{main22}.
\begin{theorem}
Assume that (A1)-(A7) and (A8)-(A12) hold.  Then there exists nonnegative $S^*_{j}\in L^\infty(\Omega_j)$, $j=1,2,...,N$, such that 
\begin{eqnarray}
&&\lim_{t\rightarrow\infty} \|S_{j}(\cdot, t)- S^*_{j}\|_{p, \Omega_j}=0, \ \ \text{for any } p>1,  \ j=1, 2,..., N, \\
&&\lim_{t\rightarrow\infty} \|E_{j}(\cdot, t)\|_{\infty, \Omega_j}=0, \ j=1, 2,...,N, \\
&&\lim_{t\rightarrow\infty} \|I_{j}(\cdot, t)\|_{\infty, \Omega_j}=0, \ j=1, 2,...,N,\\
&&\lim_{t\rightarrow\infty} \|V_{i}(\cdot, t)\|_{\infty, \mathbb{R}^2}=0.
\end{eqnarray} 
Furthermore, if $\overrightarrow{C}=0$, $V_{s0}+V_{i0}$ and $S_{j0}$ are nontrivial, and  
\begin{equation}\label{beta}
\int_{\mathbb{R}^2} \beta(x)dx>\frac{\pi^2 D_{M}}{2},
\end{equation}
then $S^*_j(x)>0$ for a.e. $x\in\Omega_j$, provided $S_{j0}(x)>0$, $j=1,2,...,N$.
\end{theorem}
\begin{proof}
We only sketch the proof. By \eqref{main22}, we have 
\begin{equation}\label{sib}
\lambda \int_0^t\int_{\Omega_j} E_{j}(x, s)dxds+\delta\int_0^t\int_{\Omega_j} I_{j}(x, s)dxds\le \int_{\Omega_j} (S_{j0}(x)+E_{j0}(x)+I_{j0}(x)) dx,  \ \ j=1, 2, ..., N,
\end{equation}
which leads to
$$
\int_0^\infty \int_{\mathbb{R}^2} F(x, t) dx dt<\infty. 
$$
Similar to the proof of  Theorem \ref{theorem_asymptotic}, we can prove 
$$
\lim_{t\rightarrow\infty} \|V_i(\cdot, t)\|_{\infty, \mathbb{R}^2}=0.
$$
By the second equation of \eqref{main22}, we have 
\begin{eqnarray*}
E_j(x, t)&=&e^{-\lambda t} E_{j0}(x, t)+\int_0^t e^{-\lambda (t-s)} \sigma_j S_j(x, s)V_i(x, s)ds  \\
&\le& M_j e^{-\lambda t}+M_j\int_0^t e^{-\lambda (t-s)} \|V_i(\cdot, s)\|_{\infty, \mathbb{R}^2}ds,
\end{eqnarray*}
for some positive constant $M_j$,  $j=1, 2,..., N$. Noticing 
$$
\lim_{t\rightarrow\infty} \int_0^t e^{-\lambda (t-s)} \|V_i(\cdot, s)\|_{\infty, \mathbb{R}^2}ds = \lim_{t\rightarrow\infty} \frac{\int_0^t e^{\lambda s} \|V_i(\cdot, s)\|_{\infty, \mathbb{R}^2}ds}{e^{\lambda t}}=\lim_{t\rightarrow\infty} \frac{e^{\lambda t} \|V_i(\cdot, t)\|_{\infty, \mathbb{R}^2}}{\lambda e^{\lambda t}}=0,
$$
we have 
$$
\lim_{t\rightarrow\infty} \|E_j(\cdot, t)\|_{\infty, \Omega_j}=0, \ \ j=1,2,...,N.
$$
Similarly, by the third equation of \eqref{main22}, we have
$$
\lim_{t\rightarrow\infty} \|I_j(\cdot, t)\|_{\infty, \Omega_j}=0, \ \ j=1,2,...,N.
$$
Since $\frac{\partial S_j}{\partial t}\le 0$, there exists nonnegative $S_j^*\in L^\infty(\Omega_j)$ such that $S_j(x, t)\rightarrow S_j^*(x)$ as $t\rightarrow\infty$ for all $x\in\Omega_j$, $j=1,2,...,N$. By Lebesgue Theorem, $S_j(\cdot, t)\rightarrow S_j^*$ as $t\rightarrow\infty$ in $L^1(\Omega_j)$. Noticing the boundedness of $S_j$, we have $S_j(\cdot, t)\rightarrow S_j^*$ as $t\rightarrow\infty$ in $L^p(\Omega_j)$ for any $p>1$, $j=1,2,...,N$.
The same as the proof of Theorem \ref{theorem_asymptotic}, we can show
\begin{equation*}
\int_0^\infty \int_{\Omega_j} V_i(x, t)dxdt<\infty, 
\end{equation*}
which means 
\begin{equation*}
\int_0^\infty  V_i(x, t) dt<\infty, \ \ \text{a.e.}\  x\in\Omega_j, \ \ j=1, 2,...,N,
\end{equation*}
By the first equation of \eqref{main22}, we have
$$
S_j(x, t)=S_{j0}(x) e^{-\int_0^t \sigma(x) V_i(x, \tau)d\tau},
$$
which implies $S_j^*(x)>0$ for a.e. $x\in\Omega_j$ provided that $S_{j0}(x)>0$, $j=1,2,...,N$.
\end{proof}

\section{An Application to Bluetongue Disease}

We illustrate the model in Section 3 with numerical simulations of bluetongue disease in sheep. 
Bluetongue disease is a non-contagious viral disease transmitted via bites of midges of the genera {\it Culicoides imicoides}, 
{\it Culicoides variipennis}, and other culicoides species carrying Bluetonge virus (BTV) to domestic and wild ruminants. Disease transmission follows a criss-cross  pattern with BTV infected midge vectors infecting uninfected host ruminants and BTV infected host  ruminants transferring the disease to uninfected  vector midges. Although a variety of ruminants, including cattle, deer, goats, dromedaries, and antelopes, can contract these diseases, our focus will be on  sheep. In sheep the effects of bluetongue disease can be devastating with high rates of morbidity and high rates of mortality \cite{Wik}. Bluetongue disease can have major negative impact on the sheep industry, as losses can accrue from reduced wool and meat production. 

A variety of mathematical models of bluetongue epidemics have been developed, including stochastic event-based probabilistic models \cite{Graesboll}, \cite{Kelso}, \cite{Sumner}, discrete time stage structured models \cite{White}, data-based atmospheric dispersion model \cite{Agren}, ordinary differential equations models \cite{CharronSeegers}, \cite{Gubbins}, ordinary functional differential equations models \cite{Gourley}, partial differential equations models with diffusion, but without advection \cite{CharronKluiters}, and partial differential equations models with diffusion and advection, but only vectors \cite{Fernandez}. We will use the vector-host diffusion and advection terms of the model in Section 3 to focus on the spatial propagation of a bluetongue epidemic by short range and long range movement of BTV infected midges.

Adult midges are approximately $1-3 \, mm$ long, and easily transported by winds \cite{Carpenter}, \cite{Fernandez}, \cite{SellersPedgley},  \cite{SellersGibbs}. It has been observed that in the absence of strong winds, adult midges typically remain within a developmental habitant range of approximately $5 \, km$ radius \cite{Pioz}, \cite{PurseCarpenter}. It has also been observed that strong wind-facilitated dispersal of midges can be hundreds of miles \cite{Brenner}, \cite{Pioz}, \cite{PurseCarpenter}, \cite{Sedda}, \cite{Sumner}, \cite{Wilson}. We will assume that midges are transported both by short-range wind movement of a few kilometers, and by semi-passive long-range wind-aided movement of hundreds of kilometers.

Bluetonge  disease is typically seasonal in regions in which frosts kill the adult midges. A controversy exists concerning the disease survival between seasons in such regions, since adult infected midges typically do not survive more than 2 or 3 months \cite{PurseCarpenter}, \cite{Wilson}. Some hypothetical explanations are the following \cite{Maclachlan}, \cite{PurseBethan}, \cite{Wik}, \cite{Wilson}: (1) a few BTV infected midges survive mild winters by locating indoors, (2) some BTV infected sheep may have chronic or latent infections over a winter,  and (3) BTV infected midges can migrate long distances from warmer temperate regions with year-round epidemics \cite{PurseCarpenter}, \cite{Sedda}.

In our numerical simulations, we will vary the advection parameter ${\vec C}$ that corresponds to the long-range directed movement of midges.  The spatial units are kilometers and the time units are months. The sheep subregions are $\Omega_1$, 
$\Omega_2$, and $\Omega_3$, which are circular regions with radii of approximately  $5 \, km$, centered at $(25 \, km, 30 \, km)$,  $(50 \, km, 30 \, km)$, and $(125 \, km, 30 \, km)$, respectively. The uninfected midges are assumed to be uniformly distributed throughout the entire region of the epidemic setting \cite{Cuellar}, \cite{Guis}. This uniform distribution of uninfected midges  is not altered significantly by the epidemic outbreak. 

The initial conditions, which are spatially normally distributed, are the same for all simulations. At time 0,  BTV infected midges are only present in  $\Omega_1$. In the simulations, at time 0, the infection breaks out in $\Omega_1$, but is not present in 
$\Omega_2$ or $\Omega_3$. The midge infection rate and the sheep infection rate are assumed to be the same in $\Omega_1$,$\Omega_2$, and $\Omega_3$. The diffusion term in the simulations corresponds to local short-range movement and the advection term corresponds to wind-directed midge movement in the $x$-direction. The simulations have the same parameters, except for the advection coefficient ${\vec C}$. The initial conditions of the simulations are given in Figure \ref{figure1} and Table \ref{table:ic}. The parameters of the simulation are given in Table \ref{table:parameters}. The {\it MATHEMATICA} code for the simulations is available upon request.

\clearpage

\begin{table}[ht]
\caption{Initial Conditions for the Simulations}
\vspace{.1in}
\centering 
\begin{tabular}{c c c}
\hline
\hline                      
Symbol 
& Value
& Total Number
\\ [0.5ex]
\hline
\hline
\vspace{.1in}
$S_1(x, y, 0)$ & $30.0 \exp \bigg[1 - \bigg( \frac{(25-x)^2}{2} - \frac{(30-y)^2}{2} \bigg)\bigg]$ &  512 
\\
$E_1(x,y, 0) = I_1(x,y, 0)$ & $0.01 \times S_1(0,x,y)$  &  5 
\\ 
$S_2(x,y, 0)$ (1st simulation) & $31.0 \exp \bigg[1 - \bigg(\frac{50-x)}{2} - \frac{(30-y)^2}{2} \bigg)\bigg]$ &  530 
\\
$S_2(x,y, 0)$ (2nd simulation) & $31.0 \exp \bigg[1 - \bigg(\frac{125-x)}{2} - \frac{(30-y)^2}{2} \bigg)\bigg]$ &  530 
\\

$E_2(x,y, 0) = I_2(x,y, 0)$ & $0.0$  &  0
 \\
$V_u(x,y, 0)$ & $1000$ per $km^2$  & \cite{Cuellar},\cite{Guis}
\\
$V_i(x,y, 0)$ & $1.0 \times I_0(x,y, 0)$  & 5
\\ [1ex]
\hline
\end{tabular}
\label{table:ic}
\end{table}

\begin{table}[ht]
\caption{Parameters for the Simulations}
\vspace{.1in}
\centering 
\begin{tabular}{c c c c}
\hline
\hline                       
Symbol 
& Meaning
& Interpretation
& Value
\\ [0.5ex]
\hline
\hline
$\beta$ & midges birth rate & 1 per month per adult & 1.0 \cite{PurseCarpenter},\cite{Wilson}
\\
m & midges death rate & 1 month lifespan &1/1000 \cite{PurseCarpenter},\cite{Wilson}
\\
$\sigma_1$ & host infection rate in $\Omega_1$  & per infected midge & 1.0  
\\
$\sigma_2$ & host infection rate in $\Omega_2$  & per infected midge & 1.0  
\\
$\sigma$ & midge infection rate  & per infected host & 0.005
\\
$\lambda$ & host incubation period  & 1 week & 4.0  \cite{Wik}
\\
$\delta$ & host infectious period  & 1 month & 1.0  \cite{Wik}
\\
$D$ & midge diffusion rate  & short-range movement & 1.0
\\
${\vec C}$ & midge advection rate  & long-range wind-aided movement 
& $-10.0  \left[ \begin{array}{cc}
1   \\
0   \\
\end{array} \right]$
\\
& 1st simulation in $\Omega_1$,$\Omega_2$  &  10.0 km per month, x-direction &
\\  
${\vec C}$ & midge advection rate  & long-range wind-aided movement 
& $-20.0  \left[ \begin{array}{cc}
1   \\
0   \\
\end{array} \right]$
\\
& 2nd simulation in $\Omega_1$,$\Omega_3$  &  20.0 km per month, x-direction &
\\  
\\ [1ex]
\hline
\end{tabular}
\label{table:parameters}
\end{table}

\clearpage

In the absence of the long-range movement advection term, that is, ${\vec C} = 0.0  \left[ \begin{array}{cc}
1   \\
0   \\
\end{array} 
\right]$, 
the epidemic remains in 
$\Omega_1$, and does not arrive at the sites $\Omega_2$ or $\Omega_3$ (Figures \ref{figure1}, \ref{figure2}, and \ref{figure3}).

\begin{figure}[ht]
\begin{center}
{\includegraphics[width=6.5in,height=1.6in]{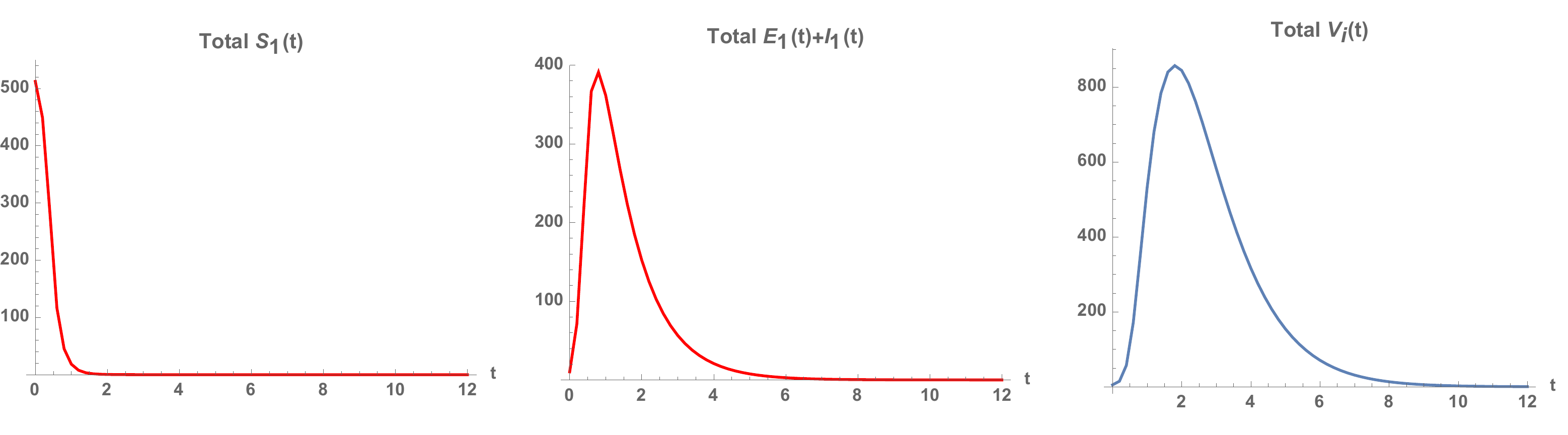}}
\end{center}
\caption{
With the advection term 
${\vec C} = {\vec {\bf 0}}$,
the population of infected sheep and the population of infected midges are effectively $0$ at time $t=12$.}
\label{figure1}
\end{figure}

\begin{figure}[ht]
\begin{center}
{\includegraphics[width=6.5in,height=2.0in]{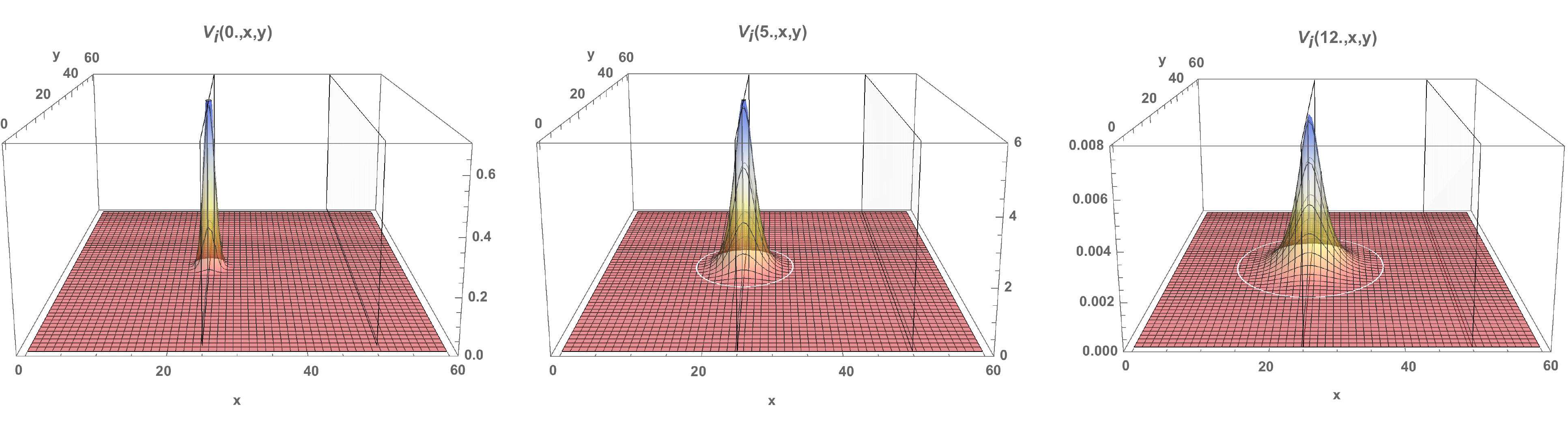}}
\end{center}
\caption{With the advection term 
${\vec C} = {\vec {\bf 0}}$, the short-range movement of midges due to diffusion, is insufficient for the population of infected midges to arrive at 
$\Omega_2$ (centered at $50 \, km$) or $\Omega_3$ (centered at $125 \, km$) before time $t=12$. Thus, the epidemic does not break out in $\Omega_2$ or $\Omega_3$.}
\label{figure2}
\end{figure}

\begin{figure}[ht]
\begin{center}
{\includegraphics[width=6.5in,height=2.0in]{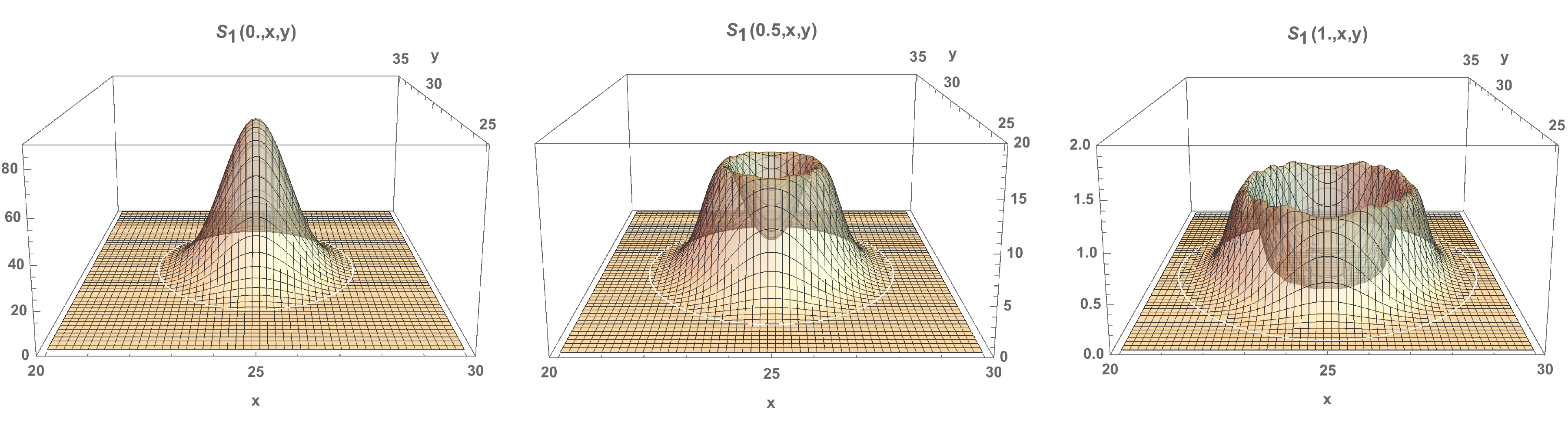}}
\end{center}
\caption{The spatial  distributions of the host population $S_1(x,y, t)$ at times $t = 0.0, \, 0.5$, and $1.0$ with the advection term 
${\vec C} = {\vec {\bf 0}}$. In the absence of long-range movement of midges due to advection, all hosts in $\Omega_1$ become infected by time $t = 2.0$.}
\label{figure3}
\end{figure}

\vspace{2.0in}

\subsection{First simulation -  lower advection coefficient}

In the first simulation, with the value of ${\vec C} =-10.0  \left[ \begin{array}{cc}
1   \\
0   \\
\end{array} 
\right]$, 
the wind-aided long-range movement of  BTV infected midges due to advection, plus the short-range movement due to diffusion, is sufficient to transport the infected midges to $\Omega_2$.  (Figures \ref{figure4},\ref{figure5}, and \ref{figure6}).

\begin{figure}[ht]
\begin{center}
{\includegraphics[width=6.5in,height=1.6in]{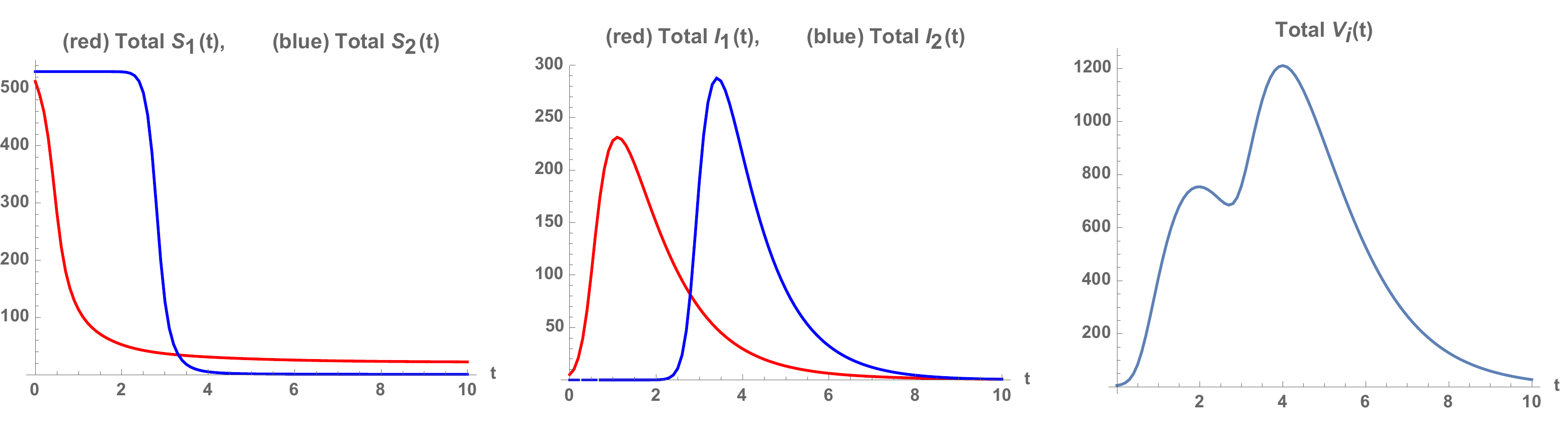}}
\end{center}
\caption{The total populations $S_1(t)$, $S_2(t)$, $I_1(t)$, $I_2(t)$, and $V_I(t)$ over 10 months. The epidemic breaks out in $\Omega_2$ at approximately 2 months. Not all hosts in the 1st site become infected.
All hosts in the second site $\Omega_2$ become infected by approximately 8 months.}
\label{figure4}
\end{figure}

\begin{figure}[ht]
\begin{center}
{\includegraphics[width=6.5in,height=3.5in]{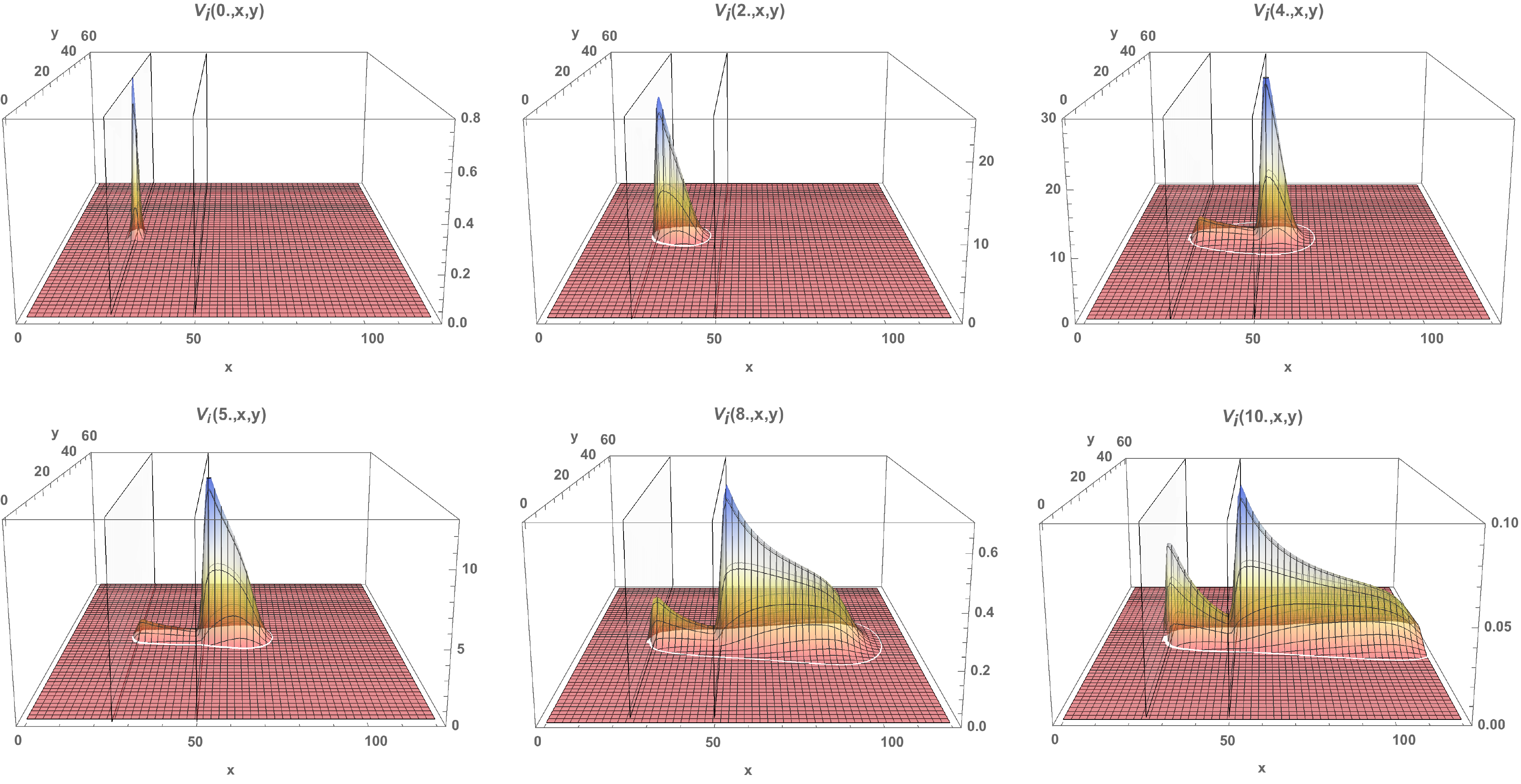}}
\end{center}
\caption{The spatial distributions of the infected vectors $V_i(x,y, t)$ at times $t = 0, 2, 4, 5 ,6, 10$ months. The long-range wind-aided advection movement plus the short-range diffusion movement of the infected vectors from the first site $\Omega_1$ is sufficient to initiate an epidemic outbreak at the second site $\Omega_2$ within approximately 2 months.}
\label{figure5}
\end{figure}

\begin{figure}[ht]
\begin{center}
{\includegraphics[width=6.5in,height=3.3in]{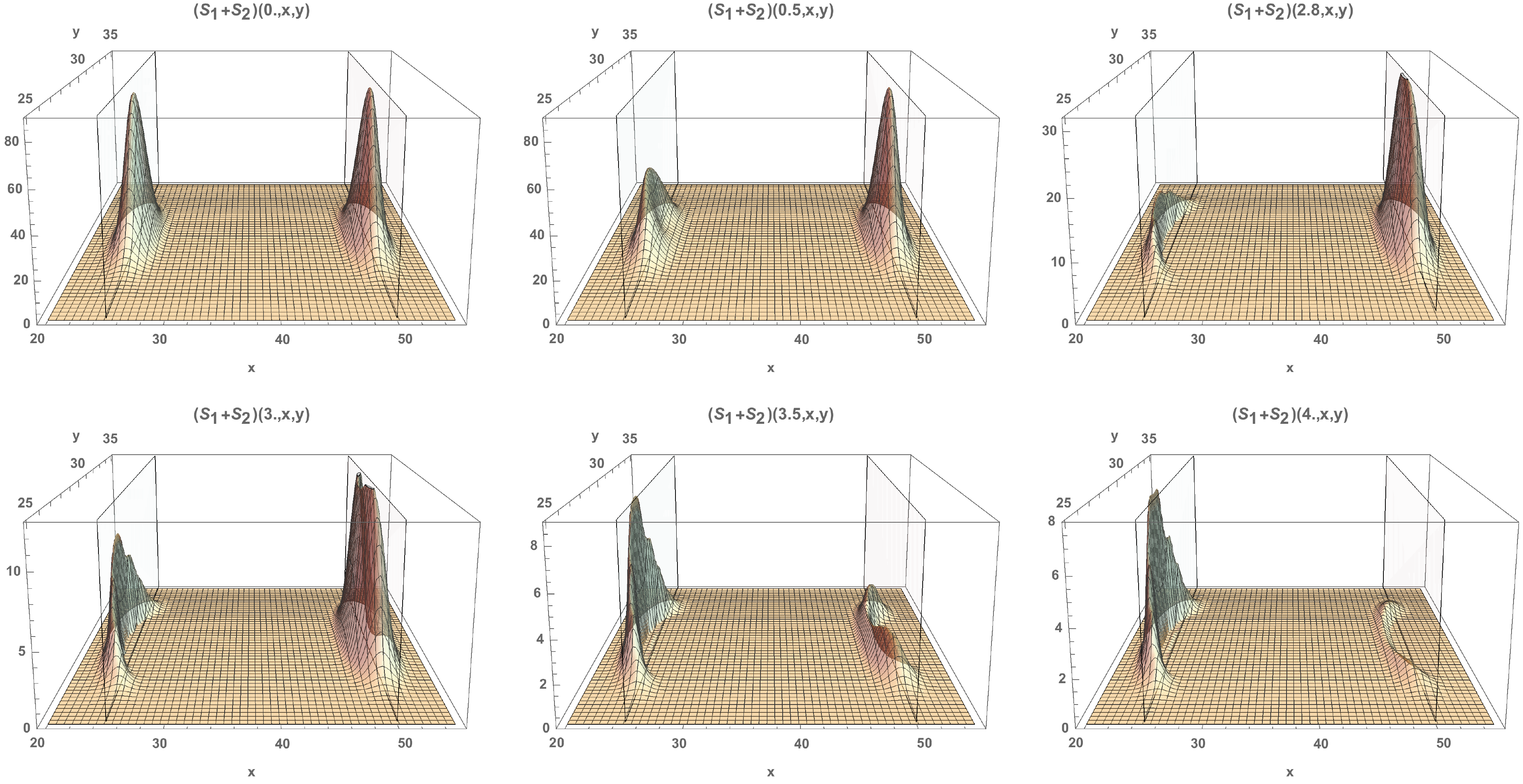}}
\end{center}
\caption{The  spatial  distributions of the host populations $S_j(x,y, t), \, j=1,2$ at times $t = 0, 0.5, 2.8, 3, 3.5, 4$ months.} 
\label{figure6}
\end{figure}

\clearpage

\subsection{Second simulation -  higher advection coefficient}

In the second simulation, with the value of ${\vec C} =-20.0  \left[ \begin{array}{cc}
1   \\
0   \\
\end{array} 
\right]$, 
the wind-aided long-range movement of  BTV infected midges due to advection, plus the short-range movement due to diffusion, is sufficient to transport the infected midges to $\Omega_3$.  (Figures \ref{figure7}, \ref{figure8}, and \ref{figure9}). 

\begin{figure}[ht]
\begin{center}
{\includegraphics[width=6.5in,height=1.6in]{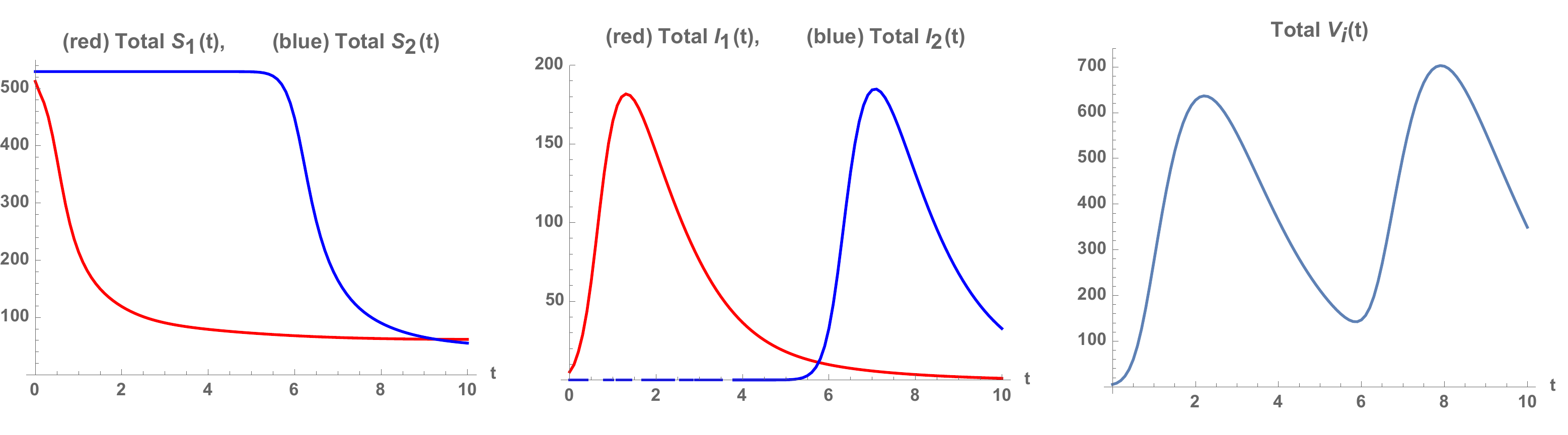}}
\end{center}
\caption{The total populations $S_1(t)$, $S_2(t)$, $I_1(t)$, $I_2(t)$, and $V_I(t)$ over 10 months. The epidemic breaks out in $\Omega_2$ at approximately 5 months. Not all hosts in the 1st site become infected and not
all hosts in the second site $\Omega_2$ become infected.}
\label{figure7}
\end{figure}

\begin{figure}[ht]
\begin{center}
{\includegraphics[width=6.5in,height=3.5in]{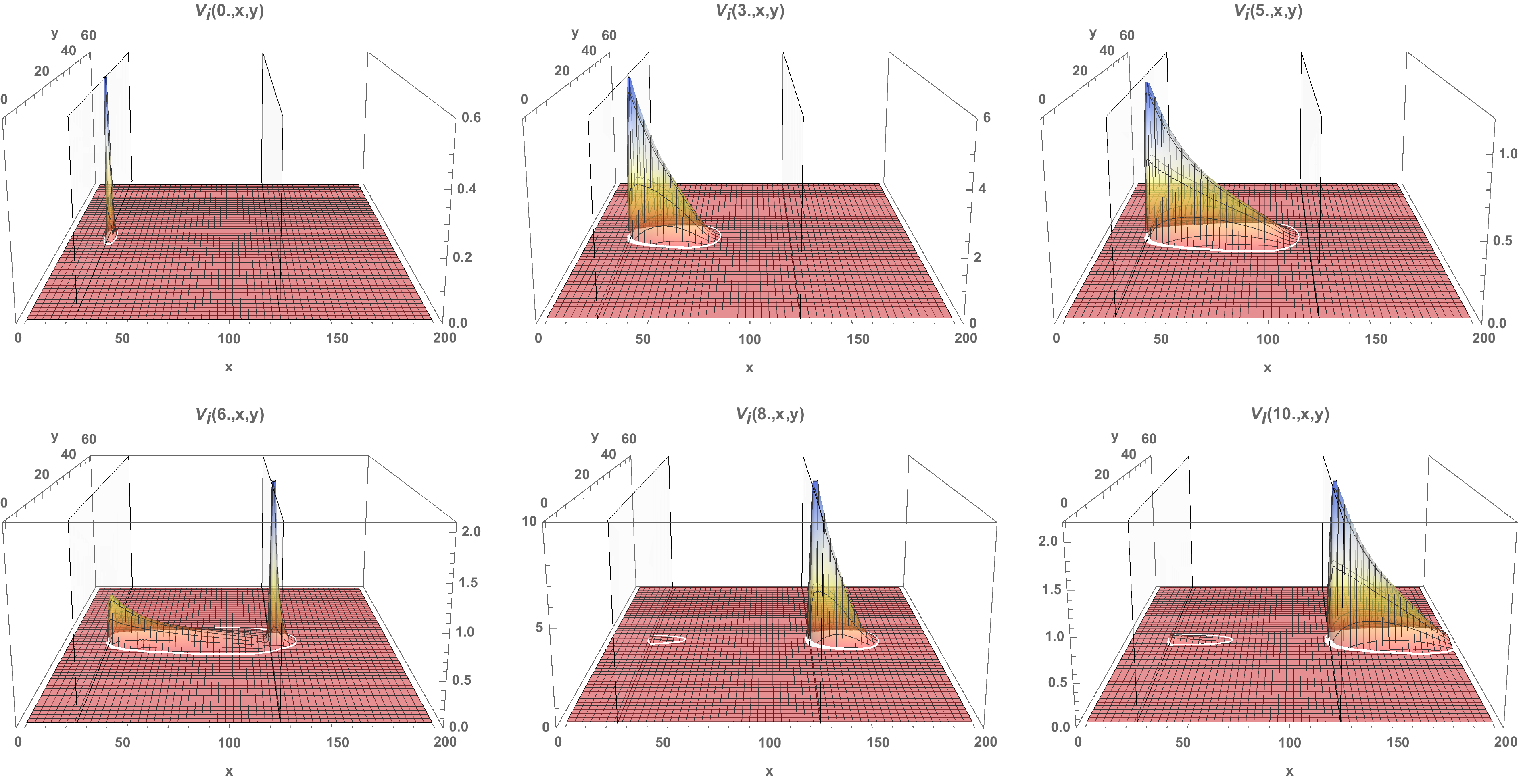}}
\end{center}
\caption{The spatial distributions of the infected vectors $V_i(x,y, t)$ at times $t = 0, 3, 5, 6 ,8,10$ months. The long-range wind-aided advection movement plus the short-range diffusion movement of the infected vectors from the first site $\Omega_1$ is sufficient to initiate an epidemic outbreak at the second site $\Omega_3$ within approximately 5 months.}
\label{figure8}
\end{figure}

\begin{figure}[ht]
\begin{center}
{\includegraphics[width=6.5in,height=3.3in]{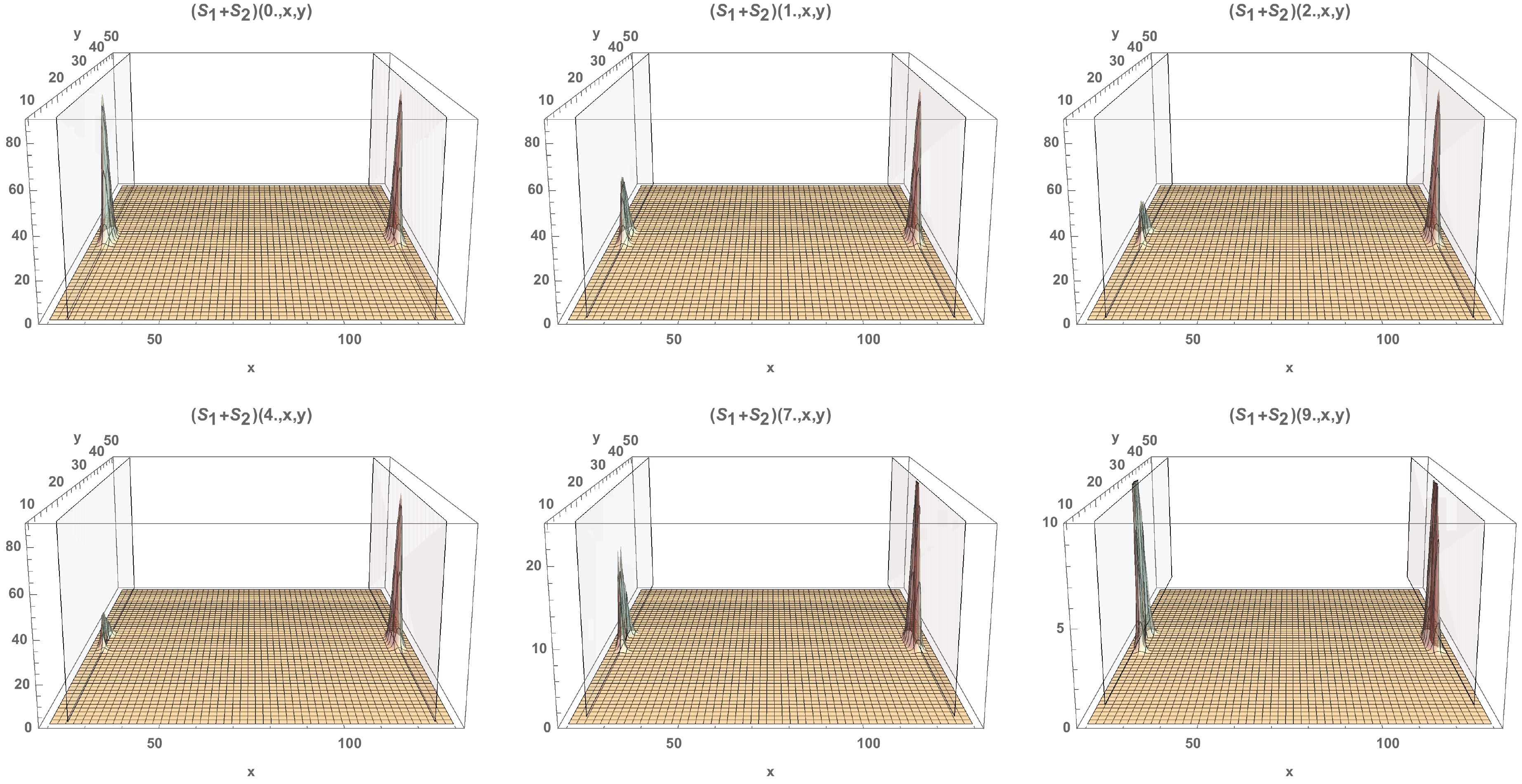}}
\end{center}
\caption{The  spatial  distributions of the host populations $S_j(x,y, t), \, j=1,2$ at times $t = 0, 1, 2, 4, 7, 9$ months.} 
\label{figure9}
\end{figure}

\clearpage

\section{Conclusions and Discussion}

We have investigated a spatial vector-host epidemic model, with hosts confined to small non-overlapping domains, and vectors moving throughout a much larger domain. The motivation of our model is to understand how an epidemic outbreak in one small region can transport to outbreaks in distant regions, in the absence of contact between hosts in these widely separately regions. The spatial movement of vectors is modeled by diffusion terms and advection terms in the model equations. The diffusion terms correspond to general short-range spatial movement and the advection terms corresponds to long-range spatial movement in a specified direction. 

We have analyzed the dynamics of the model and characterized the behavior of solutions over time. Numerical simulations illustrate how the bluetongue disease can spread from one sheep heart to other geographically separated herds . In these simulations the transport of the disease from an outbreak location to a distant location is dependent upon the magnitude of the advection term. The interpretation of the advection term is wind-aided movement of infected midges, which can be carried to distant uninfected sheep, if the wind-aided movement is sufficiently strong, but not so strong that it disperses the infected midges to values too low at out-lying sites. Our simulations have illustrated our model with three host subregions. In reality, there are a very large number of subregions in a much larger region of inhabitation of the midge population. These multiple subregions allow successive subregion-to-subregion inter-transport of infected midges by long-range movement, as represented by the simulations with only two subregions.

Our model is a simplified formulation of the  biological processes in many respects. The model is formulated as a system of continuum partial differential  equations, which relate parameter values of the model equations over time. These parameters capture average values of the dynamical processes, with possibly wide ranges of values represented by these averages,  The reality of the epidemic processes is extremely complex, and is dependent on an extreme variation in the dynamical processes. Our simplified models, however, capture the essential elements of this class of host-vector epidemics, and provide insight into  essential epidemiological behavior.

\end{document}